\documentclass[11pt]{article}

\usepackage{amsmath,amsthm,verbatim,amssymb,amsfonts,amscd, graphicx,color, framed,enumerate}
\newcommand{\ds}{\displaystyle}
\newcommand{\R}{\mathbb R}
\newcommand{\Z}{\mathbb Z}
\usepackage[left=1in,right=1in,top=1in,bottom=1in]{geometry}
\usepackage[colorlinks]{hyperref}

\theoremstyle{plain}
\newtheorem{theorem}{Theorem}[section]
\newtheorem{corollary}{Corollary}
\newtheorem{lemma}{Lemma}[section]

\theoremstyle{definition}
\newtheorem{definition}{Definition}[section]
\newtheorem{assumption}{Assumption}

\theoremstyle{definition}
\newtheorem{example}{Example}

\def\C{\mathcal C}

\def\V{\mathcal V}

\begin{document}

\title{Results on stochastic reaction networks with non-mass action kinetics}
\author{David F.~Anderson\thanks{Department of Mathematics, University of
  Wisconsin, Madison, USA.  anderson@math.wisc.edu, grant support from NSF-DMS-1318832 and Army Research Office grant W911NF-14-1-0401.},
  \and
  Tung D.~Nguyen\thanks{Department of Mathematics, University of
  Wisconsin, Madison, USA.  nguyen34@math.wisc.edu.
}
}
\maketitle

\begin{abstract}
 	In 2010, Anderson, Craciun, and Kurtz showed that if a deterministically modeled reaction network is complex balanced, then the associated stochastic model admits a stationary distribution that is a product of Poissons \cite{ACK2010}.  That work spurred a number of followup analyses.  In 2015, Anderson, Craciun, Gopalkrishnan, and Wiuf considered a particular scaling limit of the stationary distribution detailed in \cite{ACK2010}, and proved it is a well known Lyapunov function \cite{ACGW2015}.   In 2016, Cappelletti and Wiuf showed the converse of the main result in \cite{ACK2010}: if a reaction network with stochastic mass action kinetics admits a stationary distribution that is a product of Poissons, then the deterministic model is complex balanced \cite{CW2016}.   In 2017, Anderson, Koyama, Cappelletti, and Kurtz showed that the mass action models considered in \cite{ACK2010} are non-explosive (so the stationary distribution characterizes the limiting behavior).   In this paper, we generalize each of the three followup results detailed above to the case when the stochastic model has a particular form of non-mass action kinetics. 
\end{abstract}


\section{Introduction}

In this paper we prove the existence of a stationary distribution and the non-explosivity of stochastically modeled reaction networks with a specific form of non-mass action kinetics (see \eqref{general intensity}), and generalize two results related to reaction networks with mass action kinetics to the non-mass action setting.

  In \cite{ACK2010}, it was shown that if a deterministically  modeled reaction network (with deterministic mass action kinetics) is complex balanced with complex balanced equilibrium $c\in \R_{\ge 0}^m$, then the associated  stochastic model (with stochastic mass action kinetics and the same choice of rate constants)  admits the stationary distribution
\begin{align}\label{eq:9087097}
	\pi(x) = \prod_{i = 1}^me^{-c_i} \frac{c_i^{x_i}}{x_i!}.
\end{align}
 It was also shown in \cite{ACK2010} that if the intensity functions of the stochastic model are not of mass action, but of the form
\[
 \lambda_k(x)=\kappa_k\prod_{i=1}^m \theta_i(x_i)\theta_i(x_i-1)...\theta_i(x_i-y_{ki}+1)
\]
where $\theta_i: \mathbb{Z}\to\mathbb{R}_{\geq 0}$ and $\theta_i(0) = 0$, then the stochastic model admits the stationary measure
\begin{equation}\label{eq:5678987}
	\pi(x)=\prod_{i=1}^m \frac{c_i^{x_i}}{\theta_i(1)...\theta_i(x_i)}
\end{equation}
where $c$ is  the complex balanced equilibrium of the associated deterministic mass action model with rate constants $\{\kappa_k\}$.  Related work characterizing stationary distributions for stochastic models with non-mass action kinetics can be found in \cite{AC2016}.  Later, we will show that $\pi$ in \eqref{eq:5678987} is summable under some mild growth conditions on $\theta_i$, thus it can be normalized to provide a stationary distribution. 

The main result in \cite{ACK2010} related to mass action models (i.e.~the characterization of the stationary distribution for complex balanced models given in \eqref{eq:9087097}) has been the starting point  for a number of research endeavors.  In particular,
\begin{enumerate}
\item The main result of \cite{ACKK2017} gives a condition for a stochastically modeled reaction network to be non-explosive. Applying this result for complex balanced networks, whose stationary distributions are given by \eqref{eq:9087097}, it can be concluded that for any initial distribution, complex balanced stochastic mass-action systems are non-explosive.

\item The converse of the  main result of \cite{ACK2010} related to mass action models was proven in \cite{CW2016}.  That is, it was shown that if \eqref{eq:9087097} is the stationary distribution of a stochastically modeled reaction network with mass action kinetics, then the associated deterministic model with mass action kinetics is complex balanced.

\item  In \cite{ACGW2015}, it was shown that
 a particular scaling limit of the stationary distribution \eqref{eq:9087097} is a well known Lyapunov function.  This result has been useful in the study of large deviations related to stochastic models of chemical systems \cite{ADE2017,GQ2017}.
\end{enumerate}

In this paper we  generalize the three results above to the case of stochastically modeled reaction networks with non-mass action kinetics.

The outline of the remainder of the paper is as follows. In Section \ref{sec:2}, we review the relevant definitions and mathematical models for chemical reaction networks. In Section \ref{sec:3}, we state the previous results on reaction networks that were briefly described above. Finally, in Section \ref{sec:4} we provide our main results, which provide the existence of a stationary distribution and the non-explosivity of stochastically modeled reaction networks with non-mass action kinetics and generalize the previous findings on mass action systems to the non-mass action settings.

\section{Chemical Reaction Networks}
\label{sec:2}

\subsection{Reaction networks and key definitions}
We consider a set of $m$ species $\{S_1,S_2,..,S_m\}$ undergoing a finite number of reaction types enumerated by $k \in \{1,\dots,K\}$. For the $k^{th}$ reaction type we denote by $y_k$ and  $y_k' \in \mathbb{Z}^m_{\geq 0}$ the vectors representing the number of molecules of each species consumed and created in one instance of the reaction. For example, the reaction $S_1+S_2\to 2S_2$ has $y_k=(1,1)$ and $y_k'=(0,2)$, if the system has 2 species. The vectors $y_k$ and $y_k'$ are called \textit{complexes} of the system. $y_k$ is called  the \textit{source complex} and $y_k'$ is called the \textit{product complex}. A complex can be both a source complex and a product complex. 
\begin{definition}
Let $\mathcal{S} =\{S_1,...,S_m\}$, $\mathcal{C}=\cup_k\{y_k,y_k'\}$, and $\mathcal{R}=\cup_k\{y_k\to y_k'\}$ be the sets of species, complexes and reactions respectively. The triple $\{\mathcal{S,C,R}\}$ is called a reaction network.
\end{definition}

To each reaction network $\{\mathcal{S,C,R}\}$, there is a unique directed graph constructed as follows. The nodes of the graph are the complexes. A directed edge is placed from $y_k$ to $y_k'$ if and only if there is a reaction $y_k \to y_k'$. Each connected component is called a \textit{linkage class} of the graph. We denote by $\ell$ the number of linkage class.

\begin{definition}
A reaction network $\{\mathcal{S,C,R}\}$ is called \textit{weakly reversible} if the associated directed graph is strongly connected.
\end{definition}

\begin{definition}
The linear subspace $S=\text{span}_k\{y_k'-y_k\}$ generated by all reaction vectors is called the \textit{stoichiometric subspace} of the network. For $c\in \mathbb{R}^m_{\geq 0}$ we say $c+S=\{x\in \mathbb{R}^m |x=c+s$ for some $s\in S\}$ is a \textit{stoichiometric compatibility class}, $(c+S)\cap \mathbb{R}^m_{\geq 0}$ is a \textit{non-negative stoichiometric compatibility class}, and $(c+S)\cap \mathbb{R}^m_{> 0}$ is a \textit{positive stoichiometric compatibility class}. Denote $dim(S)=s$.
\end{definition}
Finally, we provide the definition of the \textit{deficiency} of a reaction network \cite{Feinberg1979}.
\begin{definition}
The \textit{deficiency} of a chemical reaction network $\{\mathcal{S,C,R}\}$ is $\delta=|\mathcal{C}|-\ell-s$, where $|\mathcal{C}|$ is the number of complexes, $\ell$ is the number of linkage classes, and $s$ is the dimension of the stoichiometric subspace of the network.
\end{definition}

\subsection{Dynamical system models}
\subsubsection{Stochastic model}
The most common stochastic model for a reaction network $\{\mathcal{S,C,R}\}$ treats the system as a continuous time Markov chain whose state at time $t$,  $X(t)\in \Z^m_{\ge 0}$, is a vector giving the number of molecules of each species present with each reaction modeled as a possible transition for the chain. The model for the $k^{th}$ reaction is determined by the source and product complexes of the reaction, and a function $\lambda_k$ of the state that gives the transition intensity, or rate, at which the reaction occurs. In the biological and chemical literature, transition intensities are referred to as propensities. 

Given that the $k^{th}$ reaction happens at time $t$, the state is updated by the addition of the reaction vector $y_k'-y_k$,
\begin{equation*}
X(t)=X(t-)+y_k'-y_k.
\end{equation*}
A common choice for the intensity functions $\lambda_k$ is to assume the system satisfies the stochastic version of \textit{mass action kinetics}. In this case, the functions have the form
\begin{equation}\label{intensity}
\lambda_k(x)=\kappa_k \prod_{i=1}^m \frac{x_i!}{(x_i-y_{ki})!} 1_{\{x_i\geq y_{ki}\}}
\end{equation} 
where $\kappa_k>0$ is  called the \textit{rate constant}. Under the assumption of mass action kinetics and a non-negative initial condition, it follows that the dynamics of the system is confined to the particular non-negative stoichiometric compatibility class determined by the initial value $X(0)$, namely $X(t)\in (X(0)+S)\cap \mathbb{R}^m_{\geq 0}$.

Simple book-keeping implies that $X(t)$ satisfies
\[
	X(t) = X(0) + \sum_k R_k(t) (y_k' - y_k),
\]
where $R_k(t)$ gives the number of times reaction $k$ has occurred by time $t$.    Kurtz showed that $X$ can be represented as the solution to the stochastic equation 
\begin{equation}\label{eq:RTC}
X(t)=X(0)+\sum_kY_k(\int_0^t \lambda_k(X(s))ds)(y_k'-y_k),
\end{equation}
where the $Y_k$ are independent unit-rate Poisson process \cite{Kurtz80}.

Another way to characterize the models of interest is via Kolmogorov's forward equation, termed the chemical master equation in the biology and chemistry literature, which describes how the distribution of the process changes in time.  Letting $p_\mu(x,t)$ give the probability that $X(t) = x$ assuming an initial distribution of $\mu$, the forward equation is
\begin{equation*}
		\frac{d}{dt} p_\mu(x,t) = \sum_k \lambda_k(x-(y_k' - y_k)) p_\mu(x-(y_k' - y_k),t) - \sum_k \lambda_k(x) p_\mu(x,t).
\end{equation*}
Constant solutions to the forward equation, i.e.~those satisfying 
\begin{equation*}
\sum_k \pi(x-y_k'+y_k)\lambda_k(x-y_k'+y_k)=\pi(x)\sum_k\lambda_k(x)
\end{equation*}
are stationary measures for the process, and if they are summable they can be normalized to give a stationary distribution.  Assuming the associated stochastic model is non-explosive, stationary distributions characterize the long-time behavior of the stochastically modeled system.

Finding  stationary distributions is in general a difficult task. However, as discussed in the Introduction and in later sections, an explicit form for $\pi(x)$ can be found when the associated deterministic mass action model has a complex balance equilibrium.

\subsubsection{Deterministic model}

Under an appropriate scaling limit (the classical scaling detailed in Section \ref{sec:classical}) the continuous time Markov chain model of the previous section becomes
\begin{equation}\label{det}
x(t)=x(0)+\sum_k(\int_0^t f_k(x(s))ds)(y_k'-y_k)
\end{equation} 
where 
\begin{equation}\label{det mak}
f_k(x)=\kappa_kx^{y_k}
\end{equation}
where $\kappa_k>0$ is the rate constant, and where  for two vectors $u,v\in \mathbb{R}^m_{\geq 0}$ we denote $u^v = \prod_i u_i^{v_i}$ with the convention $0^0=1$.  Later, we will also utilize the notation $uv$ for the vector whose $i$th component is $u_iv_i$.

 We say that the deterministic system \eqref{det} has \textit{deterministic mass action kinetics} if the rate functions $f_k$ have the form \eqref{det mak}. The system \ref{det} is equivalent to the system of ODEs 
\begin{equation}\label{det ODE}
\dot{x}=\sum_k\kappa_kx^{y_k}(y_k'-y_k).
\end{equation}
The trajectory with initial condition $x_0$ is confined to the non-negative stoichiometric compatibility class $(x_0+S)\cap \mathbb{R}^m_{\geq 0}$.

Some mass action  systems have complex balanced equilibria \cite{Horn1972,HJ1972}, which has been shown to play an important role in many biological mechanisms \cite{CLWBNE2012, Gnacadja2009, KZO2012, Sontag2001}. An equilibrium point $c$ is said to be complex balanced if for all $z\in \mathcal{C}$, we have
\begin{equation}
\sum_{k: \ y_k'=z}\kappa_kc^{y_k}=\sum_{k: \ y_k=z}\kappa_kc^{y_k}
\end{equation}
where the sum on the left is over reactions for which $z$ is the product complex and the sum on the right is over reactions for which $z$ is the source complex.

In \cite{HJ1972} it was shown that if there exists a complex balanced equilibrium $c\in \mathbb{R}^m_{\geq 0}$ then
\begin{enumerate}
\item There is one, and only one, positive equilibrium point in each positive stoichiometric compatibility class.
\item Each such equilibrium point is complex balanced.
\item Each such complex balanced equilibrium point is locally asymptotically stable relative to its stoichiometric compatibility class. 
\end{enumerate}
In \cite{C2016}, a proof is presented showing global stability relative to the stoichiometric compatibility class.  Because of the above, we say that a system is complex balanced if it admits a complex balanced equilibrium.

The condition of a reaction network being complex balanced can be checked more conveniently by using a classical result \cite{Feinberg87}.  See also \cite{Feinberg1979, Gun2003}. 
\begin{theorem}
If the reaction system is weakly reversible and has a deficiency of zero, then for any choice of rate constants $\{\kappa_k\}$ the deterministically modeled system with mass action kinetics is complex balanced.
\end{theorem}

\section{Previous Results}\label{sec:3}
\subsection{Product-form stationary distribution}
As mentioned in the previous section, complex balanced systems deserve special attention, and thus characterizations of complex balanced systems are of interest. The following Theorem in \cite{ACK2010} provides an explicit form for the stationary distribution of complex balanced systems.

\begin{theorem}\label{prod}
Let $\{\mathcal{S,R,C}\}$ be a reaction network. Suppose that when modeled deterministically with mass action kinetics and rate constants $\{\kappa_k\}$ the system is complex balanced with a complex balanced equilibrium $c\in \mathbb{R}^m_{\geq 0}$. Then the stochastically modeled system with intensities \eqref{intensity}, with the same rate constants $\{\kappa_k\}$, admits the stationary distribution 
\begin{equation}\label{prodform}
\pi(x)= \prod^m_{i=1} \frac{c_i^{x_i}}{x_i!}e^{-c_i}, \ \ \ \ \ x\in \mathbb{Z}^m_{\geq 0}.
\end{equation}
\end{theorem}
See also \cite{ACKK2017}, which shows that these systems are non-explosive, implying $\pi$ yields the limiting distributions of the process.

In the case when the stochastic model does not have mass action kinetics, \cite{ACK2010} also provides an extended result. In particular, \cite{ACK2010} considers generalized intensity functions as mentioned in several past papers \cite{Kelly1979, Whittle1986},
\begin{equation}\label{general intensity}
\lambda_k(x)=\kappa_k\prod_{i=1}^m \theta_i(x_i)\cdots\theta_i(x_i-y_{ki}+1),
\end{equation}
where $\kappa_k$ are positive rate constants, $\theta_i: \mathbb{Z} \to \mathbb{R}_{\geq 0}$, and $\theta_i(x)=0$ if $x\leq 0$. The functions $\theta_i$ should be thought of as the "rate of association" of the $i^{th}$ species \cite{Kelly1979}. For a system with intensity functions \eqref{general intensity}, the product form stationary distribution is quite similar to the one in Theorem \ref{prod}.

\begin{theorem}\label{general prod}
Let $\{\mathcal{S,R,C}\}$ be a reaction network. Suppose that when modeled deterministically with mass action kinetics and rate constants $\{\kappa_k\}$ the system is complex balanced with a complex balanced equilibrium $c\in \mathbb{R}^m_{\geq 0}$. 
Then the stochastically modeled system with general intensity functions \eqref{general intensity}, with the same rate constants $\{\kappa_k\}$, admits the stationary measure 
\begin{equation}\label{general prodform}
\pi(x)= \prod_{i=1}^m \frac{c_i^{x_i}}{\theta_i(1)\cdots\theta_i(x_i)}, \ \ \ \ \ x\in \mathbb{Z}^m_{\geq 0}.
\end{equation}
\end{theorem}

In Section \ref{sec:4}, we will show that $\pi$ in \eqref{general prodform} is summable under some mild growth condition on $\theta_i$, and thus it can be normalized to a stationary distribution.

The result in Theorem \ref{prod} is a cornerstone for later research in both application and theory. It has been used in various studies in applied and computational biology such as tumor growth and invasion \cite{KSO2011}, phosphorylation systems \cite{CLWBNE2012}, studies of the chemical master equation \cite{KKNS2014}, etc. It is also the background for many theoretical developments in the field of chemical reaction networks \cite{ACGW2015, CW2016}.
Interestingly, it has been proven  that the converse is also true. 
\begin{theorem}[\cite{CW2016}]\label{eq:thm4}
Let $\{\mathcal{S,R,C}\}$ be a reaction network and consider the stochastically modeled system with rate constants $\{\kappa_k\}$ and mass action kinetics  \eqref{intensity}.  Suppose  that  for some $c \in \mathbb{R}^m_{\ge 0}$ the stationary distribution for the stochastic  model is \eqref{prodform}.  Then $c$ is a complex balanced equilibrium for the associated deterministic model with mass action kinetics and rate constants $\{\kappa_k\}$.
\end{theorem}
%

\subsection{Non-equilibrium potential and Lyapunov functions}
Another interesting result comes from the scaling behavior of the stationary distribution for complex balanced system. We first provide a key definition.
\begin{definition}
Let $\pi$ be a probability distribution on a countable set $\Gamma$ such that $\pi(x)>0$ for all $x\in\Gamma$. The \textit{non-equilibrium potential} of the distribution $\pi$ is the function $\phi_\pi: \Gamma \to \mathbb{R}$, defined by
\begin{equation*}
\phi_\pi(x)=-\ln(\pi(x)).
\end{equation*}
\end{definition}
In \cite{ACGW2015} it was shown that under an appropriate scaling, the limit of the non-equilibrium potential of the stationary distribution of a complex balanced system converges to a certain well-known Lyapunov function.
\begin{definition}
Let $E \subset \mathbb{R}^m_{\geq 0}$ be an open subset of $\mathbb{R}^m_{\geq 0}$ and let $f: \mathbb{R}^m_{\geq 0} \to \mathbb{R}$. A function $\mathcal{V}: E\to \mathbb{R}$ is called a \textit{Lyapunov function} for the system $\dot x = f(x)$ at $x_0 \in E$ if $x_0$ is an equilibrium point for $f$, that is $f(x_0)=0$, and 
\begin{enumerate}
\item $\mathcal{V}(x)>0$ for all $x\neq x_0, x\in E$ and $\mathcal{V}(x_0)=0$.
\item $\nabla\mathcal{V}(x)\cdot f(x)\leq 0$, for all $x\in E$, with equality if and only if $x=x_0$, where $\nabla\mathcal{V}$ denotes the gradient of $\mathcal{V}$.
\end{enumerate}
\end{definition}
In particular, the non-equilibrium potential of the stationary distribution in \eqref{prodform} converges to the usual Lyapunov function of Chemical Reaction Network Theory
\begin{equation}\label{Lyapunov}
\mathcal{V}(x)=\sum_i x_i(\ln(x_i)-\ln(c_i)-1)+c_i.
\end{equation}
The next section discusses the scaling  in which the convergence happens. It is called the classical scaling in the literature.

\subsubsection{The classical scaling}
\label{sec:classical}

Here we present a brief introduction to the classical scaling.  For more detailed discussions, see \cite{AK2011, AK2015, Kurtz78}.

Let $|y_k|=\sum_i y_{ki}$ and let $V$ be the volume of the system times Avogadro's number. Suppose $\{\kappa_k\}$ are the rate constants for the stochastic model. We defined the scaled rate constants as follows
\begin{equation}\label{scaled rate}
\kappa_k^V = \frac{\kappa_k}{V^{|y_k|-1}}
\end{equation}
and denote the scaled intensity function for the stochastic model by
\begin{equation}\label{scaled intensity}
\lambda_k^V(x)=\frac{\kappa_k}{V^{|y_k|-1}} \prod^m_{i=1}\frac{x_i!}{(x_i-y_{ki})!}.
\end{equation}
Note that if $x \in\Z^m_{\ge 0}$ gives the counts of the different species, then $\tilde x :=V^{-1} x$ gives the concentrations in moles per unit volume.  Then, by standard arguments
\begin{equation*}
\lambda_k^V(x) \approx V\kappa_k\prod^m_{i=1} \tilde x_i^{y_{ki}} = V\lambda_k(\tilde x)
\end{equation*}
where the final equality defines $\lambda_k$ and justifies the definition of deterministic mass action kinetics.

Denote the stochastic process determining the counts by $X^V(t)$, then normalizing the original process $X^V$ by V and defining $\bar X^V:=\ds\frac{X^V}{V}$ gives us
\begin{equation*}
\bar X^V(t) \approx \bar X^V(0) + \sum_k \frac{1}{V}Y_k(V\int_0^t\lambda_k(\bar X^V(s))ds)(y_k'-y_k),
\end{equation*}
where we are utilizing the representation \eqref{eq:RTC}.
Since the law of large numbers for the Poisson process implies $V^{-1}Y(Vu)\approx u$, we may conclude that a good approximation to the process $\bar X^V$ is the function $x=x(t)$ defined as the solution to the ODE
\begin{equation*}
\dot x = \sum_k \kappa_kx^{y_k}(y_k'-y_k),
\end{equation*}
which is exactly \eqref{det ODE}.

A corollary of Theorem \ref{prod} gives us the stationary distribution for the classically scaled system

\begin{theorem}\label{prod scaling}
Let $\{\mathcal{S,R,C}\}$ be a reaction network. Suppose that when modeled deterministically with mass action kinetics and rate constants $\{\kappa_k\}$ the system is complex balanced with a complex balanced equilibrium $c\in \mathbb{R}^m_{\geq 0}$.
For $V>0$, let $\{\kappa_k^V\}$ satisfy \eqref{scaled rate}. Then the stochastically modeled system on $\Z^d_{\ge 0}$ with rate constants $\{\kappa_k^V\}$ and intensity functions \eqref{scaled intensity} admits the stationary distribution
\begin{equation}\label{scaled SD1}
\pi^V(x)=  \prod_{i=1}^m \frac{(Vc_i)^{x_i}}{x_i!}e^{-Vc_i}, \ \ \ \ \ x \in \Z^m_{\ge 0}.
\end{equation}
\end{theorem}
An immediate implication of Theorem \ref{prod scaling} is that a stationary distribution for the scaled model $\bar X^V$ is
\begin{equation}\label{scaled SD2}
\tilde \pi^V(\tilde x^V)=\pi^V(V\tilde x^V), \ \ \ for \ \tilde x^V \in \frac{1}{V} \mathbb{Z}^d_{\geq 0}.
\end{equation}
\subsubsection{Convergence of the non-equilibrium potential}
The main finding in \cite{ACGW2015} is concerned with the scaling limit of the stationary distribution $\tilde \pi^V$ of \eqref{scaled SD2}.
\begin{theorem}\label{NEP}
Let $\{S, C, R\}$ be a  reaction network and let  $\{\kappa_k\}$ be a choice of rate constants.  Suppose that, modeled deterministically,  the system is complex balanced.  For  $V>0$, let $\{\kappa^V_k\}$ be related to $\{ \kappa_k\}$ via \eqref{scaled rate}. Fix a sequence of points  $\tilde x^V\in \frac1V \Z^d_{\ge 0}$  for which  $\lim_{V\to \infty}\tilde {x}^V = \tilde x \in \R^d_{>0}$.  Further let $c$ be the unique complex balanced equilibrium within the positive stoichiometric compatibility class of $\tilde x$. 
  
    Let $\pi^V$ be given by \eqref{scaled SD1} and let $\tilde \pi^V$ be as in \eqref{scaled SD2}, then
  \begin{equation*}
  	\lim_{V\to \infty}\left[- \frac{1}{V}\ln(\tilde \pi^V\!(\tilde x^V)) \right]= \V(\tilde x),
  \end{equation*}
  where $\V$ is the Lyapunov function for the ODE model satisfying \eqref{Lyapunov}. 
\end{theorem}

\section{Main results for networks with general kinetics}\label{sec:4}
In this section, we first show that for stochastically modeled reaction networks with non-mass action kinetics defined via \eqref{general intensity} whose associated mass action system is complex balanced, the stationary measure \eqref{general prodform} can be normalized to yield a stationary distribution.  We further show that these stochastic models are non-explosive.    We then extend Theorems \ref{eq:thm4} and  \ref{NEP} from Section \ref{sec:3} to the non-mass action case. 

\subsection{Existence of a stationary distribution and non-explosivity of non-mass action systems}
We begin with a theorem proving that the stochastic models considered  in Theorem \ref{general prod} are positive recurrent when only mild growth conditions are placed on the functions $\theta_i$.

\begin{theorem}\label{main theorem 2}
Let $\{S,C,R\}$ be a reaction network with rate constants $\{\kappa_k\}$. Suppose that when modeled deterministically, the associated mass action system is complex balanced with equilibrium $c\in \mathbb{R}^m_{>0}$.   
Suppose that $\theta_i$ and $\lambda_k$ satisfy the conditions in and around \eqref{general intensity}.  Moreover, suppose that for each $i$ we have $\lim_{x\to \infty} \theta_i(x) = \infty$.  Then,
\begin{enumerate}
\item the  measure $\pi$ given in \eqref{general prodform} is summable over $\Z^m_{\ge 0}$, and a stationary distribution exists for the stochastically modeled process, and moreover
\item the stochastically modeled process is non-explosive.
\end{enumerate}
\end{theorem}
\begin{proof} 
We first show that $\pi$ is summable over $\Z^m_{\ge 0}$.
We have
\begin{align*}
\sum_{x \in \Z_{\ge 0}^m}  \pi(x) = \sum_{x\in\Z_{\ge 0}^m}\prod_{i=1}^m \frac{c_i^{x_i}}{\theta_i(1)\cdots\theta_i(x_i)}=\prod_{i=1}^m (\sum_{x_i\in\Z_{\ge 0}} \frac{c_i^{x_i}}{\theta_i(1)\cdots\theta_i(x_i)})
\end{align*}
so long as each sum in the final expression is finite. Thus it is sufficient to prove that $\ds\sum_{x\in\Z_{\ge 0}} \frac{c_i^{x}}{\theta_i(1)\cdots\theta_i(x)}$ is finite for each   $i$. By the ratio test
\begin{equation*}
\lim_{x\to\infty} \frac{c_i^{x+1}}{\theta_i(1)\cdots\theta_i(x+1)} \cdot \left(\frac{c_i^{x}}{\theta_i(1)\cdots\theta_i(x)}\right)^{-1}=\lim_{x\to\infty}\frac{c_i}{\theta_i(x+1)}=0<1
\end{equation*}
where the last equality is due to the assumption that $\lim_{x\to \infty} \theta_i(x) = \infty$.  Hence the sum is convergent.

We turn to showing that the process is non-explosive. From \cite{ACKK2017}, to show that the process is non-explosive, it is sufficient to show
\begin{equation*}
\sum_{x\in\Z^m_{\ge 0}} \left( \pi(x)\sum_{k}\lambda_k(x)\right) <\infty.
\end{equation*}
From \eqref{mod intensity} and \eqref{mod SD1}, we need to show
\begin{equation*}
\sum_{x\in\Z^m_{\ge 0}} \left( \prod_{i=1}^m\frac{c_i^{x_i}}{\theta_i(1)\cdots\theta_i(x_i)}\sum_{k}\prod_{i=1}^m\theta_i(x_i)\cdots \theta_i(x_i-y_{ki}+1) \right)<\infty.
\end{equation*}
Let $s_i=\max_k \{y_{ki}\}$, where the max is over all source complexes,  and let  $R$ be the number of reactions. Let $n_i>s_i$ be such that $\theta_i(x_i)>1,\cdots,\theta_i(x_i-s_i+1)>1$ for all $x_i>n_i$. Then
\begin{align*}
\sum_{x\in\Z^m_{\ge 0};x_i>n_i}\prod_{i=1}^m&\frac{c_i^{x_i}}{\theta_i(1)\cdots\theta_i(x_i)}\sum_{k}\prod_{i=1}^m\theta_i(x_i)\cdots \theta_i(x_i-y_{ki}+1)\\
&<\sum_{x\in\Z^m_{\ge 0};x_i>n_i}\prod_{i=1}^m\frac{c_i^{x_i}}{\theta_i(1)\cdots\theta_i(x_i)}R\prod_{i=1}^m\theta_i(x_i)\cdots\theta_i(x_i-s_i+1)\\
&=\sum_{x\in\Z^m_{\ge 0};x_i>n_i}\prod_{i=1}^m\frac{Rc_i^{x_i}}{\theta_i(1)\cdots\theta_i(x_i-s_i)}\\
&<C\sum_{x\in\Z^m_{\ge 0};x_i>n_i}\prod_{i=1}^m\frac{c_i^{x_i-s_i}}{\theta_i(1)\cdots\theta_i(x_i-s_i)}<\infty
\end{align*}
where $C=R\max_{i=1}^m\{c_i^{s_i}\}$, and the last inequality follows from part 1. Thus the process is non-explosive.
\end{proof}

\subsection{Generalization of Theorem \ref{eq:thm4}}
We are set to provide the next theorem of the current paper, which is the converse statement of Theorem \ref{general prod} and generalizes Theorem \ref{eq:thm4} from \cite{CW2016}.   In the theorem below, we assume $\lim_{x\to \infty} \theta_i(x) = \infty$ for each $i$.  In Corollary \ref{cor:main}, we generalize the result to allow $\lim_{x\to \infty} \theta_i(x) \in \{0,\infty\}$ for each $i$.

\begin{theorem}\label{main theorem 1}
Let $\{\mathcal{S,R,C}\}$ be a reaction network and consider the stochastically modeled system with rate constants $\{\kappa_k\}$ and intensity functions \eqref{general intensity}.  Suppose  that $\lim_{x\to \infty} \theta_i(x) = \infty$  for each $i =1,\dots,m$ and that for some $c \in \mathbb{R}^m_{\ge 0}$ a stationary measure for the stochastic  model  satisfies \eqref{general prodform}.  Then $c$ is a complex balanced equilibrium for the associated deterministic model with mass action kinetics and rate constants $\{\kappa_k\}$.
\end{theorem}

\begin{proof}
By assumption, we have that $\pi$ satisfies 
	\begin{equation*}
	\sum_k \pi(x+y_k-y_k')\lambda_k(x+y_k-y_k')=\pi(x)\sum_k\lambda_k(x).
	\end{equation*}
Plugging \eqref{general intensity} and \eqref{general prodform} into this equation yields
	\begin{align*}
	\sum_k  \ds\frac{c^{x+y_k-y_k'}}{\prod_{i=1}^m[\theta_i(1)\cdots\theta_i(x_i+y_{ki}-y_{ki}')]}\kappa_k \prod_{i=1}^m \theta_i(x_i+y_{ki}-y_{ki}')\cdots\theta_i(x_i-y_{ki}'+1)\\=  \ds\frac{c^x}{\prod_{i=1}^m[\theta_i(1)\cdots\theta_i(x_i)]} \sum_{i=1}\kappa_k \prod_{i=1}^m \theta_i(x_i)\cdots\theta_i(x_i-y_{ki}+1).
	\end{align*}
Canceling and moving terms when necessary, we have
	\begin{align*}
	\sum_k c^{y_k-y_k'}\kappa_k\prod_{i=1}^m \theta_i(x_i)\cdots\theta_i(x_i-y_{ki}'+1)= 
	\sum_k \kappa_k\prod_{i=1}^m \theta_i(x_i)\cdots\theta_i(x_i-y_{ki}+1).
	\end{align*}	
Enumerating the reaction on the right by their product complexes, and the reactions on the left by their source complexes, the equation above becomes
\begin{align*}
	\sum_{z\in \mathcal{C}}\prod_{i=1}^m \theta_i(x_i)\cdots\theta_i(x_i-z_i+1) \sum_{k:y_k'=z}c^{y_k-y_k'}\kappa_k= 
	\sum_{z\in \mathcal{C}} \prod_{i=1}^m \theta_i(x_i)\cdots\theta_i(x_i-z_i+1)\sum_{k:y_k=z} \kappa_k.
	\end{align*}	
Since the above holds for all $x\in \mathbb{Z}^m_{\ge 0}$, the two sides are equal as functions. Hence, if the functions in the set 
\begin{equation}\label{theta}
\{\prod_{i=1}^m\theta_i(x_i)\cdots\theta_i(x_i-z_i+1)\}_{z\in \mathcal{C}}
\end{equation} 
are linearly independent, then we must have 
\begin{align*}
\sum_{k:y_k'=z}c^{y_k-y_k'}\kappa_k=\sum_{k:y_k=z} \kappa_k,
\end{align*}
which is the condition for the associated mass action system to be complex balanced.

Thus it remains to show that the functions in the set \eqref{theta}  are linearly independent.  We will prove that the functions are linearly independent by induction on the number of species, and begin with the base case $m=1$, which we provide as a lemma.
\begin{lemma}\label{lemma 1}
When $m=1$ (i.e.~the system has only one species) the functions in the set \eqref{theta} are linearly independent.
\end{lemma}
\begin{proof}[Proof of Lemma \ref{lemma 1}]
Let $\mathcal{C}=\{z_1, \dots,z_n\}$ ordered so that $z_i < z_{i+1}$ for each $i =1,\dots,n-1$.
Suppose, in order to find a contradiction, the  functions in the set \eqref{theta} are linearly dependent.  Then there exist $\alpha_1,\cdots,\alpha_r \in \R$ with $r\le n$ and $\alpha_r \ne 0$, such that
	\begin{align}\label{l1}
	\alpha_1\theta(x)\cdots\theta(x-z_1+1)+\cdots+\alpha_r\theta(x)\cdots\theta(x-z_r+1)=0, \quad \text{for all} \quad x \in \R.
	\end{align}
Let $M=\ds\frac{|\alpha_1|}{|\alpha_r|}+\cdots+\frac{|\alpha_{r-1}|}{|\alpha_r|}$.   Since $\theta(x)\to\infty$, as $x\to \infty$, we can find an $N>0$ such that $\forall x>N$, we have $\theta(x-z_r+1)>M$ and $\theta(x),\dots,\theta(x-z_r+1)\geq 1$. In this case,
	\begin{align*}
|\alpha_r\theta(x)\cdots\theta(x-z_r+1)|&>M |\alpha_r|\theta(x)\cdots\theta(x-z_r+2)\\
&= \left(\frac{|\alpha_1|}{|\alpha_r|}+\cdots+\frac{|\alpha_{r-1}|}{|\alpha_r|}\right)|\alpha_r|\theta(x)\cdots\theta(x-z_r+2)\\
&= |\alpha_1|\theta(x)\cdots\theta(x-z_r+2)+\cdots+|\alpha_{r-1}|\theta(x)\cdots\theta(x-z_r+2)\\
&\geq |\alpha_1|\theta(x)\cdots\theta(x-z_1+1)+\cdots+|\alpha_{r-1}|\theta(x)\cdots\theta(x-z_{r-1}+1).
	\end{align*}	
This contradicts \eqref{l1}. Therefore, \eqref{theta} must be linearly independent.
\end{proof}

We turn to the inductive step.  Thus, we now assume that functions of the form \eqref{theta} for distinct complexes $z$ are linearly independent when there are $m-1$ species.  We must show that this implies linear independence when there are  $m$ species.

Enumerate the complexes as $\mathcal{C}=\{z^1,z^2,\ldots,z^n\}$. Suppose that there are $\alpha_1, \dots, \alpha_n$ for which 
\begin{align}\label{l21}
	\alpha_1\prod_{i=1}^m\theta_i(x_i)\cdots\theta_i(x_i-z_i^1+1)+\ldots+\alpha_n\prod_{i=1}^m\theta_i(x_i)\cdots\theta_i(x_i-z_i^n+1)=0, \quad \text{for all} \quad x \in \R^m.
\end{align}
We will show that each $\alpha_i = 0$.

First note that we can not have $z^1_i=z^2_i=\ldots=z^n_i$ for each $i = 1, \dots,m$, for otherwise all the complexes are the same.  Thus, and without loss of generality, we assume that not all of the $z_1^k$  are equal.  In particular, we will assume that $z_1^1,\dots,z_1^n$ consists of $p$ distinct values with $2\le p \le n$.   We will also assume that the complexes are ordered so that the first $r_1$ terms of $z_1^k$ are the same, the second $r_2$ terms are the same, etc.  That is, 
\begin{equation}\label{eq:5789888}
	z_1^1 = \ldots = z_1^{r_1},  \quad z_1^{r_1 + 1} = \ldots=z_1^{r_2},  \dots  , z_1^{r_{p-1}+1} = \ldots = z_1^{n}.
\end{equation}

We now consider  the left hand side of \eqref{l21} as a function of $x_1$ alone.  For $j = 1,\dots,p$, we define 
\[
	f_j(x_1)=\theta_1(x_1)\cdots\theta_1(x_1-z_1^{r_j}+1).
\]
By Lemma \ref{lemma 1}, the functions $f_j,$ for $j=1,\dots,p$, are linearly independent. Combining similar terms in \eqref{l21} we have
\begin{align}
	&f_1(x_1)[\alpha_1\prod_{i=2}^m\theta_i(x_i)\cdots\theta_i(x_i-z_i^1+1)+\ldots+\alpha_{r_1}\prod_{i=2}^m\theta_i(x_i)\cdots\theta_i(x_i-z_i^{r_1}+1)]+\label{eq:firstbracketedterm}\\
	&f_2(x_1)[\alpha_{r_1+1}\prod_{i=2}^m\theta_i(x_i)\cdots\theta_i(x_i-z_i^{r_1+1}+1)+\ldots+\alpha_{r_2}\prod_{i=2}^m\theta_i(x_i)\cdots\theta_i(x_i-z_i^{r_2}+1)]+\notag\\
	&\vdots\notag\\
	&+f_{p}(x_1)[\alpha_{r_{p-1}+1}\prod_{i=2}^m\theta_i(x_i)\cdots\theta_i(x_i-z_i^{r_{p-1}+1}+1)+\ldots+\alpha_{n}\prod_{i=2}^m\theta_i(x_i)\cdots\theta_i(x_i-z_i^{n}+1)]=0.\notag
\end{align}
From the independence of the $f_j$, it must be the case that each bracketed term above is zero.  

Without loss of generality, we just consider the first bracketed term in \eqref{eq:firstbracketedterm}:
\begin{align}\label{eq:76868}
\alpha_{1}\prod_{i=2}^m\theta_i(x_i)\cdots\theta_i(x_i-z_i^{1}+1)+\ldots+\alpha_{r_1}\prod_{i=2}^m\theta_i(x_i)\cdots\theta_i(x_i-z_i^{r_1}+1),
\end{align}
which we know is equal to zero.    The goal now is to  apply our inductive hypothesis to conclude that each of $\alpha_1,\dots,\alpha_{r_1}$ is equal to zero.  

%
For each of $k = 1,\dots,r_1$, we let $\tilde z^k = (z^k_2,\dots,z^k_m)$. Then each term in the sum \eqref{eq:76868} is a function on $\R^{m-1}$ of the general form \eqref{theta} with new complexes $\tilde z^k \in \R^{m-1}$. To use the inductive hypothesis, we  must argue  that the $\tilde z^k$ are distinct.    Consider, for example, the first two terms: $\tilde z^1$ and $\tilde z^2$.  By \eqref{eq:5789888}, we know that $z^1_1=z^2_1$; that is, the coefficient of species 1 for the two complexes are the same.  If we also had  $\tilde z^1 = \tilde z^2$, then all the coefficients of the species would be the same for the two complexes, contradicting the fact that they are distinct complexes (i.e.~$z^1 \ne z^2$).  Hence, it must be that $\tilde z^1 \ne \tilde z^2$.
Thus, by the inductive hypothesis, all the terms of the sum \eqref{eq:76868} are linearly independent, and $\alpha_1=\cdots=\alpha_{r_1}=0$. Repeating this argument for the other bracketed terms completes the proof.

We have proven the independence of \eqref{theta} in all cases, which completes the proof of the theorem.
\end{proof}


The following relaxes the condition in Theorem \ref{main theorem 1} that the limit of the functions $\theta_i$ must be infinity.

\begin{corollary} \label{cor:main}
Let $\{\mathcal{S,R,C}\}$ be a reaction network and consider the stochastically modeled system with rate constants $\{\kappa_k\}$ and intensity functions \eqref{general intensity}.  Suppose  that $\lim_{x\to \infty} \theta_i(x) \in\{0, \infty\}$  for each $i =1,\dots,m$ and that for some $c \in \mathbb{R}^m_{\ge 0}$ a stationary measure for the stochastic  model  satisfies \eqref{general prodform}.  Suppose further that $\theta_i(x)>0$ for $x$ large enough.  Then $c$ is a complex balanced equilibrium for the associated deterministic model with mass action kinetics and rate constants $\{\kappa_k\}$.
\end{corollary}
\begin{proof}
Without loss of generality, assume $\lim_{x\to \infty} \theta_i(x)= 0$ for $i\le  \ell$ and $\lim_{x\to \infty} \theta_i(x) = \infty$ for $i \ge \ell + 1.$  The proof is the same as that of Theorem \ref{main theorem 1} in that we must prove the linear independence of the functions in \eqref{theta}. Let 
\begin{align}\label{cor 11}
	\alpha_1\prod_{i=1}^m\theta_i(x_i)\cdots\theta_i(x_i-z_i^1+1)+\ldots+\alpha_n\prod_{i=1}^m\theta_i(x_i)\cdots\theta_i(x_i-z_i^n+1)=0.
\end{align}
For $x$ large enough that $\theta_i(x) >0$,  let $\phi_i(x)=\ds\frac{1}{\theta_i(x)}$ for each  $i \le \ell$.  Then  we have $\lim_{x\to \infty} \phi_i(x) = \infty$. Now \eqref{cor 11} becomes
\begin{align}\label{cor 12}
	\frac{\alpha_1\prod_{i=\ell+1}^m\theta_i(x_i)\cdots\theta_i(x_i-z_i^1+1)}{\prod_{i=1}^\ell \phi_i(x_i)\cdots\phi_i(x_i-z_i^1+1)}+\ldots+\frac{\alpha_n\prod_{i=\ell+1}^m\theta_i(x_i)\cdots\theta_i(x_i-z_i^n+1)}{\prod_{i=1}^\ell \phi_i(x_i)\cdots\phi_i(x_i-z_i^n+1)}=0.
\end{align}
Let $w_k=\max_{1\le j \le n}\{z_k^j\}$. Then from \eqref{cor 12} we have
\begin{align*}
	&\alpha_1\prod_{i=1}^\ell\phi_i(x_i-z_i^1)\cdots\phi_i(x_i-w_i)\prod_{i=\ell+1}^m\theta_i(x_i)\cdots\theta_i(x_i-z_i^1+1)+\\
	&\ldots+\alpha_n\prod_{i=1}^\ell\phi_i(x_i-z_i^n)\cdots\phi_i(x_i-w_i)\prod_{i=\ell+1}^m\theta_i(x_i)\cdots\theta_i(x_i-z_i^n+1)=0.
\end{align*}
This is similar to the set-up of Theorem \ref{main theorem 1} (since each $\phi_i(x) \to \infty$, as $x \to \infty$) and we can conclude $\alpha_1=\ldots=\alpha_n=0$ and complete the proof.
\end{proof}

\subsection{Generalization of Theorem \ref{NEP}}
This section is concerned  with the convergence of the non-equilibrium potential of the stationary distribution of systems with general kinetics, under some appropriate scaling.  In particular, we would like to have a similar result as Theorem \ref{NEP} for the case of general kinetics.  One difficulty that arises is that the classical scaling is not, in general, appropriate for our purposes. This is illustrated by the example below. 

\begin{example}
Consider the reaction network with one species A and reactions given by
\begin{equation*}
\emptyset \rightleftarrows A
\end{equation*}
with the intensity function given by \eqref{general intensity}, where the rate constants are $\kappa_{\emptyset\to A}=\kappa_{A\to\emptyset}=1$ and $\theta(x)=x^2$. Consider the process under the classical scaling. We obtain a stationary distribution for the scaled model $\tilde X^V$ from Theorem \ref{prod scaling} and \eqref{scaled SD2}:
\begin{equation*}
\tilde\pi^V(\tilde x^V)=\pi^V(V\tilde x^V)=\frac{1}{M}\frac{(Vc)^{V\tilde x^V}}{\theta(1)\cdots\theta(V\tilde x^V)}=\frac{1}{M}\frac{(Vc)^{V\tilde x^V}}{((V\tilde x^V)!)^2},  \qquad \tilde x^V \in \frac1V \Z_{\ge 0}.
\end{equation*}
We consider the limiting behavior of the non-equilibrium potential $\ds-\frac{1}{V}\ln(\tilde\pi^V(\tilde x^V))$, as $V\to \infty$.  Using Stirling's approximation,
\begin{align*}
-\frac{1}{V}\ln(\tilde \pi^V(\tilde x^V))&=-\frac{1}{V}\ln\left(\frac{1}{M}\frac{(Vc)^{V\tilde x^V}}{((V\tilde x^V)!)^2}\right)\\
&=-\frac{1}{V}(-\ln M+V\tilde x^V\ln V+V\tilde x^V\ln c - 2\ln ((V\tilde x^V)!))\\
&\approx -\frac{1}{V}(-\ln M+V\tilde x^V\ln V+ V\tilde x^V\ln c -2V\tilde x^V\ln V\tilde x^V+2 V\tilde x^V)\\
&=-\frac{1}{V}(-\ln M+V\tilde x^V\ln c - 2V\tilde x^V\ln \tilde x^V -V\tilde x^V\ln V +2 V\tilde x^V).
\end{align*}
We need to estimate $M$ when $V\to \infty$. From Lemma \ref{lemma 3} in the Appendix, we have
\begin{align*}
\ln M&=\ln \sum_{x\in \mathbb{Z}_{\ge 0}}\frac{(Vc)^x}{(x!)^2}\\
&\approx 2(Vc)^{1/2} + a\ln (Vc) + b,
\end{align*}
for some constants $a,b \in\R$.
Thus
\begin{align*}
-\frac{1}{V}\ln(\tilde \pi^V(\tilde x^V))&\approx -\frac{1}{V}(-2(Vc)^{1/2}-a\ln (Vc)-b +V\tilde x^V\ln c-2V\tilde x^V\ln \tilde x^V-V\tilde x^V\ln V+2 V\tilde x^V)\\
&=2\frac{c^{1/2}}{V^{1/2}}+\frac{a\ln (Vc)}{V}+\frac{b}{V}-\tilde x^V\ln c +2\tilde x^V\ln \tilde x^V+\tilde x^V\ln V-  2\tilde x^V.
\end{align*}
Clearly, $\ds\lim_{V\to \infty}-\frac{1}{V}\ln(\tilde\pi^V(\tilde x^V))=\infty$, and we do not have  convergence of the non-equilibrium potential under the classical scaling. \hfill $\square$
\end{example}

With the above example in mind, we provide an alternative scaling.
\subsubsection{The modified scaling}
Define $|y_k|=\sum_i y_{ki}$ and let $V$ be a scaling parameter.  For each reaction $y_k\to y_k'$ let $\kappa_k$ be a positive parameter.  We now define
the rate constant for $y_k\to y_k'$ as 
\begin{equation}\label{mod rate}
\kappa_k^V = \frac{\kappa_k}{V^{d\cdot y_k-1}}
\end{equation}
where the parameter $d$ is a vector to be chosen (they will depend upon the limiting values $\lim_{x\to \infty} \theta_i(x)$). Note that the classical scaling is the case when $d =(1,1,\ldots,1)$. Then we define the scaled intensity function
\begin{equation}\label{mod intensity}
\lambda_k^V(x)=\frac{\kappa_k}{V^{d\cdot y_k-1}}\prod_{i=1}^m\theta_i(x_i)\cdots\theta_i(x_i-y_{ki}+1),
\end{equation}
where, as usual,  $\theta_i: \mathbb{Z} \to \mathbb{R}_{\geq 0}$, and $\theta_i(x)=0$ if $x\leq 0$.
\begin{theorem}\label{mod prodform}
Let $\{S,C,R\}$ be a reaction network with rate constants $\{\kappa_k\}$. Suppose that when modeled deterministically, the associated mass action system is complex balance with equilibrium $c\in \mathbb{R}^m_{>0}$. For some V, let $\{\kappa_k^V\}$ be related to the $\{\kappa_k\}$ via \eqref{mod rate}. Then the stochastically modeled system with scaled intensity function \eqref{mod intensity} has  stationary measure
\begin{equation}\label{eq:modStationaryMeasure}
\pi_{*}^V(x) =  \ds\prod_{i=1}^m \frac{(V^{d_i}c_i)^{x_i}}{\theta_i(1)\cdots\theta_i(x_i)}, \quad \text{ where } \quad x \in \Z^m_{\ge 0}.
\end{equation}
If \eqref{eq:modStationaryMeasure} is summable, then a normalizing constant $M$ can be found so that 
\begin{equation}\label{mod SD1}
\pi^V(x)= \frac1M\ds\prod_{i=1}^m \frac{(V^{d_i}c_i)^{x_i}}{\theta_i(1)\cdots\theta_i(x_i)}, \quad \text{ where } \quad x \in \Z^m_{\ge 0},
\end{equation}
is a stationary distribution.
\end{theorem}
\begin{proof}
The proof is similar to that of Theorem \ref{prod}, as found in \cite{ACK2010}, except care must be taken to ensure that the terms associated with the scaling parameter $V$ cancel appropriately.  
\end{proof}

Let $X^V$ be the process associated with the intensities \eqref{mod intensity} and let $\tilde X^V = V^{-1} X^V$ be the scaled process.  By Theorem \ref{mod prodform}, we have that for  $x^V \in \frac1V \Z^m_{\ge 0}$
\begin{equation}\label{measure11111}
\tilde \pi_*^V(\tilde x^V) = \pi_*^V (V\tilde x^V), 
\end{equation}
is a stationary measure for the scaled process.  If \eqref{measure11111} is summable, which is ensured by Theorem \ref{main theorem 2} so long as $\theta_i(x) \to \infty$ as $x\to \infty$, then
\begin{equation}\label{mod SD2}
\tilde \pi^V(\tilde x^V) = \pi^V (V\tilde x^V), 
\end{equation}
is a stationary distribution.
In the next section, we consider the the limiting behavior of $-\frac1V \ln(\tilde \pi^V(\tilde x^V))$ as $V \to \infty$ for a class of $\theta_i$.

\subsubsection{Limiting behavior of $-\frac1V \ln(\tilde \pi^V(\tilde x^V))$}

We make the following assumption on the functions $\theta_i$.
\begin{assumption}\label{class}
We assume that (i) $\theta_i: \mathbb{Z} \to \mathbb{R}_{\geq 0}$, (ii)  $\theta_i(x)=0$ if $x\leq 0$, and (iii) there exists $d,A\in \mathbb{R}^m_{>0}$  such that $\lim_{x_i\to \infty} \ds\frac{\theta_i(x_i)}{x_i^{d_i}}=A_i$ for each $i$. 
\end{assumption}
Roughly speaking, this class of functions act like power functions when $x$ is large. We will utilize functions satisfying Assumption \ref{class} to build intensity functions as in \eqref{mod intensity}.
We will  show that if the deterministic mass action system is complex balanced, then  the limiting behavior of the scaled non-equilibrium potential of the stochastically modeled system with intensities \eqref{mod intensity}  is a Lyapunov function for   the ODE system
\begin{align}\label{eq:56789876}
	\dot x = \sum_k \kappa_k (Ax^d)^{y_k}(y_k' - y_k), \quad \text{ for } \quad x \in \R^m_{\ge 0}.
\end{align}
where we recall that $Ax^d$ is the vector with $i$th component $A_ix_i^{d_i}$. This result therefore  generalizes Theorem \ref{NEP} (which is Theorem 8 in \cite{ACGW2015}).

\begin{lemma}
Let $\{S,C,R\}$ be a reaction network with rate constants $\{\kappa_k\}$. Suppose that when modeled deterministically, the associated mass action system is complex balanced with equilibrium $c\in \mathbb{R}^m_{>0}$.   
Let $d,A\in \R^m_{>0}$.   Then the system \eqref{eq:56789876} is complex balanced with equilibrium vector $\tilde c$ satisfying
\[
	\tilde c_i = \left(\frac{c_i}{A_i}\right)^{1/d_i}.
\]
\begin{proof}
The proof consists of verifying that for each $z\in \C$,
\[
	\sum_{k: y_k =z} \kappa_k (A \tilde c^d)^{y_k} = \sum_{k: y_k' =z} \kappa_k (A \tilde c^d)^{y_k},
\]
where the sum on the left consists of those reactions with source complex $z$ and the sum on the right consists of those with product complex $z$.  This is immediate from the definition of $\tilde c$.
\end{proof}
\end{lemma}

%
%
%
%
%

We now turn to the scaled models, and prove that the properly scaled non-equilibrium potential converges to a Lyapunov function for the ODE system \eqref{eq:56789876}.

\begin{theorem}\label{main theorem 3}
Let $\{S,C,R\}$ be a reaction network with rate constants $\{\kappa_k\}$. Suppose that when modeled deterministically, the associated mass action system is complex balanced with equilibrium $c\in \mathbb{R}^m_{>0}$.   

Fix $d,A\in \R^m_{>0}$ and let $\theta_i$ be a choice of  functions satisfying Assumption \ref{class}.  For $V>0$ and the $d>0$ already selected, let $\{\kappa_k^V\}$ be related to  $\{\kappa_k\}$ as in \eqref{mod rate} and let the intensity functions for the stochastically modeled system be \eqref{mod intensity}. 
Let $\tilde \pi^V$ be the stationary distribution for the scaled process guaranteed to exist by Theorems \ref{mod prodform} and \ref{main theorem 2} and given by \eqref{mod SD1}. 

Fix a sequence of points $\tilde x^V \in \tfrac1V \Z^m_{\ge 0}$ for which $\lim_{V\to \infty} \tilde x^V =\tilde x \in \Z^m_{>0}$. 
Then 
\begin{equation}\label{eq:09878908}
\lim_{V \to \infty} [-\frac{1}{V} \ln(\tilde \pi^V(\tilde x^V))]=\mathcal{V}(\tilde x) = \sum_{i=1}^m[ \tilde x_i(d_i\ln (\tilde x_i)-\ln(c_i)-d_i+\ln(A_i))]+ \sum_{i=1}^m d_i(c_i/A_i)^{1/d_i},
\end{equation}
where $\mathcal V$ is defined by the final equality, and moreover 
 $\mathcal V$ is  a Lyapunov function for the ODE system \eqref{eq:56789876}.
\end{theorem}

Note that by taking $d=(1,..,1)$ and $A=(1,..,1)$, the limit of the $\theta_i$ in Assumption \ref{class} is simply mass action kinetics.  Hence, the main result in \cite{ACGW2015} is contained within the above theorem.

%

\begin{proof}
Using \eqref{mod SD1}  and \eqref{mod SD2} we have 
\begin{align*}
-\frac{1}{V} \ln(\tilde \pi^V(\tilde x^V)) &= -\frac{1}{V} \ln\bigg(\frac{1}{M} \prod_{i=1}^m \frac{(V^{d_i}c_i)^{V\tilde x_i^V}}{\theta_i(1)\cdots\theta_i(V\tilde x_i^V)}\bigg)\\
&=-\frac{1}{V}\bigg( -\ln M + \sum_{i=1}^m V\tilde x_i^V \ln(V^{d_i} c_i) - \sum_{i=1}^m  \ln(\theta_i(1)\cdots\theta_i(V\tilde x_i^V))\bigg)\\
&=-\frac{1}{V}\bigg(-\ln M + \sum_{i=1}^m Vd_ix_i^V \ln(V)+ \sum_{i=1}^m V\tilde x_i^V \ln(c_i)-\sum_{i=1}^m  \ln(\theta_i(1)\cdots\theta_i(V\tilde x_i^V))\bigg)\\
&=-\frac{1}{V}\bigg(-\ln M + \sum_{i=1}^m Vd_i\tilde x_i^V \ln(V)+ \sum_{i=1}^mV\tilde x_i^V \ln(c_i)-\sum_{i=1}^m \ln((V\tilde x_i^V !)^{d_i})\\
&\hspace{.2in} +\sum_{i=1}^m \ln((V\tilde x_i^V !)^{d_i})-\sum_{i=1}^m\ln(\theta_i(1)\cdots\theta_i(V\tilde x_i^V))\bigg)\\
\end{align*}
We analyze the limiting behavior of the different pieces of the last expression.

\begin{enumerate}
\item We begin with the  first term
\begin{align*}
\frac{1}{V}\ln M=\frac{1}{V}\ln\bigg(\sum_{x\in \Z^m}  \frac{(V^dc)^{x}}{\prod_{i=1}^m\theta_i(1)\cdots\theta_i(x_i)}\bigg),
\end{align*}
where $M$ is defined using \eqref{mod SD1}.
In  Lemma \ref{lemma 5} in the appendix we show that as $V\to\infty$
\begin{align}
\frac{1}{V}\ln\bigg(\sum_{x\in \Z^m}  \frac{(V^dc)^{x}}{\prod_{i=1}^m\theta_i(1)\cdots\theta_i(x_i)}\bigg) &\sim \frac{1}{V}\ln\bigg(\sum_{x\in \Z^m}  \frac{(V^dc)^{x}}{\prod_{i=1}^m A_i^{x_i}(x_i!)^{d_i}}\bigg)\notag\\
&=\frac{1}{V}\ln\bigg(\sum_{x\in \Z^m}  \frac{(V^dcA^{-1})^{x}}{\prod_{i=1}^m (x_i!)^{d_i}}\bigg),
\label{eq:567898889}
\end{align}
where by $a_V\sim b_V$, as $V\to \infty$, we mean $\lim_{V\to \infty} (a_V - b_V) = 0$.
We may then apply Lemma \ref{lemma 3} to \eqref{eq:567898889} to conclude there are constants $a,b$ such that
\begin{align*}
\frac{1}{V}\ln\bigg(\sum_{x\in \Z^m}  \frac{(V^dcA^{-1})^{x}}{\prod_{i=1}^m (x_i!)^{d_i}}\bigg)\sim \frac{1}{V}\sum_{i=1}^m( d_i(V^{d_i}c_iA_i^{-1})^{1/d_i} + a\ln(V^{d_i}c_iA_i^{-1})+b).
\end{align*}
Taking the limit $V\to \infty$, we see that only the first term remains, which yields
\begin{align*}
\lim_{V \to \infty} \frac{1}{V}\ln M = \sum_{i=1}^m d_i(c_i/A_i)^{1/d_i}.
\end{align*}

\item We use Stirling's approximation with the middle terms
\begin{align*}
-\frac{1}{V}\bigg(&\sum_{i=1}^m (Vd_i\tilde x_i^V \ln(V)+ V\tilde x_i^V \ln(c_i))-\sum_{i=1}^m \ln((V\tilde x_i^V !)^{d_i})\bigg)\\
&\sim -\frac{1}{V}\bigg(\sum_{i=1}^m V d_i\tilde x_i^V \ln(V)+ V\tilde x_i^V \ln(c_i)-\sum_{i=1}^md_i((V\tilde x_i^V)\ln(V\tilde x_i^V)-V\tilde x_i^V)\bigg)\\
&=-\frac{1}{V}\bigg(\sum_{i=1}^m V\tilde x_i^V \ln(c_i)-\sum_{i=1}^m d_i(V\tilde x_i^V)\ln(\tilde x_i^V)+\sum_{i=1}^md_i V\tilde x_i^V\bigg)\\
&=\sum_{i=1}^m d_i\tilde x_i^V\ln(\tilde x_i^V) -\sum_{i=1}^m \tilde x_i^V\ln(c_i) - \sum_{i=1}^m d_i\tilde x_i^V. 
\end{align*}
Taking the limit $V \to \infty$, and noting that $\tilde x^V \to \tilde x$, we have 
\begin{align*}
\lim_{V \to \infty}-\frac{1}{V}\bigg(&\sum_{i=1}^m (Vd_i\tilde x_i^V \ln(V)+ V\tilde x_i^V \ln(c_i))-\sum_{i=1}^m \ln((V\tilde x_i^V !)^{d_i})\bigg)\\
&=\sum_{i=1}^m d_i\tilde x_i\ln(\tilde x_i) -\sum_{i=1}^m \tilde x_i\ln(c_i) -  \sum_{i=1}^m d_i\tilde x_i.
\end{align*}
\item We turn to the final term.  By using an argument similar to  \eqref{eq:567898889}, there is a constant $C>0$ for which
\begin{align}
-\sum_{i=1}^m\frac{1}{V}[\ln((V\tilde x_i^V !)^{d_i})-\ln(\theta_i(1)\cdots\theta_i(V\tilde x_i^V))]& = -\frac{1}{V}\ln \bigg(\frac{(V\tilde x^V !)^d}{\prod_{i=1}^m\theta_i(1)\cdots\theta_i(V\tilde x_i^V)}\bigg)\notag\\
&\sim -\frac{1}{V} \ln\bigg( \frac{C}{(A)^{V\tilde x^V}}\bigg)\label{eq:55555}\\
&=-\frac{1}{V}\big(\ln C - \sum_{i=1}^mV\tilde x_i^V\ln(A_i)\big).\notag
\end{align}
 Taking the limit $V \to \infty$, and noting that $\tilde x^V \to \tilde x$, we have 
\begin{align*}
\lim_{V \to \infty} -\sum_{i=1}^m\frac{1}{V}[\ln((V\tilde x_i^V !)^{d_i})-\ln(\theta_i(1)\cdots\theta_i(V\tilde x_i^V))] = \sum_{i=1}^m\tilde x_i \ln(A_i).
\end{align*}
\end{enumerate}
Combining the three parts, we conclude \eqref{eq:09878908} holds. 
The fact that the limit is a Lyapunov function is proven in  Lemma \ref{lemma 4} below.
\end{proof}

\begin{lemma}\label{lemma 4}
The function given by \eqref{eq:09878908},
\begin{equation*}
\V(x)= \sum_{i=1}^m[ x_i(d_i\ln (x_i)-\ln(c_i)-d_i+\ln(A_i))+d_i(c_i/A_i)^{1/d_i}], \quad x \in \Z^m_{\ge 0},
\end{equation*}
is a Lyapunov function for  the system \eqref{eq:56789876}.
\end{lemma}
\begin{proof}
We have
\begin{align*}
\nabla \V(x)&=(d_1\ln(x_1)-\ln(c_1)+\ln(A_1),\ldots,d_m\ln(x_m)-\ln(c_m)+\ln(A_m)).
\end{align*}
Let $f$ be the right-hand side of \eqref{eq:56789876} and recall that $c$ is a complex balanced equilibrium of the mass action model.  We have
\begin{align*}
\nabla \V(x)\cdot f(x)&=\sum_{k}\kappa_k(Ax^d)^{y_k}\left(\ln(x^d)-\ln\left(\frac{c}{A}\right)\right)\cdot(y_k'-y_k)\\
&=\sum_k \kappa_k c^{y_k}\frac{(Ax^d)^{y_k}}{c^{y_k}}\left(\ln\left(\frac{Ax^d}{c}\right)^{y_k'}-\ln\left(\frac{Ax^d}{c}\right)^{y_k}\right)\\
&\leq \sum_k \kappa_k c^{y_k}\left(\left(\frac{Ax^d}{c}\right)^{y_k'}-\left(\frac{Ax^d}{c}\right)^{y_k}\right)\\
&=\sum_k (Ax^d)^{y_k'}\kappa_kc^{y_k-y_k'}-(Ax^d)^{y_k}\kappa_k\\
&=\sum_{z\in\C}\left[\sum_{k: y_k'=z}(Ax^d)^{y_k'}\kappa_kc^{y_k-y_k'}- \sum_{k:y_k=z}(Ax^d)^{y_k}\kappa_k\right]\\
&=\sum_{z\in\C} (Ax^d)^z\left[\sum_{k: y_k'=z} \kappa_kc^{y_k-y_k'}- \sum_{k:y_k=z} \kappa_k\right]=0,
\end{align*}
where we used the inequality $a(\ln{b}-\ln{a})\leq b-a$ and the last equation holds because $c$ is the complex balanced equilibrium for the mass-action system.
\end{proof}

\appendix
\section{Appendix: Required Lemmas}
\begin{lemma}\label{lemma 3}
Here we need to provide an asymptotic estimate as $C\to\infty$ of the form
\begin{align*}
\ln\sum_{x\in \mathbb{Z}_{\ge 0}^m} \frac{C^x}{(x!)^d}\sim \sum_m (d_iC_i^{1/d_i} + a \ln C_i + b ) 
\end{align*}
where $a,b$ are constants that do not depend on $C$.
\end{lemma}
\begin{proof}
When $m=1$, by Stirling estimation (and ignoring the  factor of $\sqrt{2\pi}$), we have
\begin{align*}
(x!)^d \sim \bigg(\sqrt{x}\frac{x^x}{e^x}\bigg)^d = \sqrt{xd}\frac{{(xd)}^{xd}}{e^{xd}d^{xd}}x^{(d-1)/2}\sim \Gamma(xd+1)\frac{x^{(d-1)/2}}{d^{xd}}\sim \frac{\Gamma(xd+1+(d-1)/2)}{d^{xd}},
\end{align*}
where the last estimation is due to the fact that $\ds\lim_{n\to\infty}\frac{\Gamma(n+\alpha)}{n^\alpha\Gamma(n)}=1$.

Thus 
\begin{align}\label{eq4545}
\sum_{x\in \mathbb{Z}}\frac{C^x}{(x!)^d}\sim\sum_{x\in \mathbb{Z}}\frac{(Cd^d)^x}{\Gamma(xd+(d+1)/2)}.
\end{align}
The asymptotic behavior of the right hand side in \eqref{eq4545} can be found in Example 2.3.1 of \cite{PW1986}. In particular, its asymptotic character is exponential since we are considering $C$ having real values only
\begin{align*}
\sum_{x\in \mathbb{Z}}\frac{(Cd^d)^x}{\Gamma(xd+(d+1)/2)} \sim \frac{1}{d}(Cd^d)^{(1-d)/2d}e^{(Cd^d)^{1/d}}
= c C^{(1-d)/2d}e^{dC^{1/d}} 
\end{align*}
where $c$ is some constant depending on $d$.
Thus, taking log we have
\begin{align*}
\ln\bigg(\sum_{x\in \mathbb{Z}} \frac{C^x}{(x!)^d}\bigg) \sim dC^{1/d} + a \ln C + b
\end{align*}
where $a,b$ are some constants depending on $d$.\\\\
When $m>1$, we have
\begin{align*}
\ln\bigg(\sum_{x\in \mathbb{Z}^m} \frac{C^x}{(x!)^d}\bigg)=\ln\bigg(\prod_m \sum_{x\in \mathbb{Z}} \frac{C_i^{x_i}}{(x_i!)^{d_i}}\bigg)\sim \sum_m (d_iC_i^{1/d_i} + a \ln C_i + b ),
\end{align*}
where we have applied the $m=1$ case in the final step.
\end{proof}

In the following lemma, for sequences $a_V$ and $b_V$ we write $a_V \sim b_V$, as $V\to \infty$, for 
\[
\lim_{V\to \infty} (a_V-b_V) = 0.
\]
  
\begin{lemma}\label{lemma 5}
Let $\theta_i$ satisfy Assumption \ref{class}.  Then for a fixed $c\in \Z^m_{>0}$,
\begin{align*}
\frac{1}{V}\ln\bigg(\sum_{x\in \Z^m}  \frac{(V^dc)^{x}}{\prod_{i=1}^m\theta_i(1)\cdots\theta_i(x_i)}\bigg) \sim \frac{1}{V}\ln\bigg(\sum_{x\in \Z^m}  \frac{(V^dc)^{x}}{\prod_{i=1}^m A_i^{x_i}(x_i!)^{d_i}}\bigg) =\frac{1}{V}\ln \bigg(\sum_{x\in \Z^m}  \frac{(V^dcA^{-1})^{x}}{\prod_{i=1}^m (x_i!)^{d_i}}\bigg),
\end{align*}
as $V\to \infty$.
\end{lemma}
\begin{proof}
Let first consider the case when $m=1$.

Let $\alpha>0$. From Assumption \ref{class} we have $\lim_{x\to\infty}\ds\frac{\theta(x)}{x^{d}}=A_i$.  Therefore, there exists an $N>0$ such that if $x>N$, then
\begin{align*}
A-\alpha A<\frac{\theta(x)}{x^{d}}<A+\alpha A
\end{align*} 
which is equivalent to
\begin{align}\label{eq1314}
1-\alpha <\frac{\theta(x)}{Ax^{d}}<1+\alpha. 
\end{align} 
Consider
\begin{align*}
\frac{1}{V}\ln\bigg(\sum_{x\in \Z}&  \frac{(V^dc)^{x}}{\theta(1)\cdots\theta(x)}\bigg)- \frac{1}{V}\ln\bigg(\sum_{x\in \Z}  \frac{(V^dc)^{x}}{ A^{x}(x!)^{d}}\bigg)\\
=&\frac{1}{V}\ln\bigg(\ds\frac{\sum_{x\leq N}  \frac{(V^dc)^{x}}{\theta(1)\cdots\theta(x)}+\sum_{x>N}  \frac{(V^dc)^{x}}{\theta(1)\cdots\theta(x)}}{\sum_{x\in \Z}  \frac{(V^dc)^{x}}{ A^{x}(x!)^{d}}}\bigg)\\
=&\frac{1}{V}\ln\bigg(\ds\frac{R+S}{T}\bigg)\\
\end{align*}
where 
\begin{align*}
R=&\sum_{x\leq N}  \frac{(V^dc)^{x}}{\theta(1)\cdots\theta(x)} =O(V^{dN})\\
T=&\sum_{x\in \Z}  \frac{(V^dc)^{x}}{ A^{x}(x!)^{d}} = O(e^{d(V^dc/A)^{1/d}})=O(e^{dV(c/A)^{1/d}}),
\end{align*}
where the sum in $R$ is over all $x\in \Z_{\ge 0}$  and the estimation of $T$ is again based on Lemma \ref{lemma 3}. 

We have
\begin{align*}
S&=\sum_{x>N}  \frac{(V^dc)^{x}}{\theta(1)\cdots\theta(x)}\\
&=\sum_{x>N}  \frac{(V^dc)^{x}}{ A^{x}(x!)^{d}}\frac{A^{x}(x!)^{d}}{\theta(1)\cdots\theta(x)}\\
&=\frac{A^{N}(N!)^{d}}{\theta(1)\cdots\theta(N)}\sum_{x>N}  \frac{(V^dc)^{x}}{A^{x}(x!)^{d}}\frac{ A(N+1)^{d_i}\cdots A(x)^{d}}{\theta(N+1)\cdots\theta(x)}\\
&=c_N \sum_{x>N}  \frac{(V^dc)^{x}}{ A^{x}(x!)^{d}}\frac{ A(N+1)^{d}\cdots A(x)^{d}}{\theta(N+1)\cdots\theta(x)}.
\end{align*}
Using \eqref{eq1314}, we have
\begin{align*}
c_N\sum_{x>N}  \frac{(V^dc)^{x}}{A^{x}(x!)^{d}}\frac{1}{(1+\alpha)^{x-N}}<S<c_N\sum_{x>N}  \frac{(V^dc)^{x}}{ A^{x}(x!)^{d}}\frac{1}{(1-\alpha)^{x-N}}.
\end{align*}
Thus 
\begin{align*}
c_N\sum_{x>N}  \frac{(V^dc(1+\alpha)^{-1})^{x}}{ A^{x}(x!)^{d}}<S<c_N\sum_{x>N}  \frac{(V^dc(1-\alpha)^{-1})^{x}}{ A^{x}(x!)^{d}}
\end{align*}
where $LHS = O(e^{dV(c/A(1+\alpha))^{1/d}})$ and $RHS=O(e^{dV(c/A(1-\alpha))^{1/d}})$ similar to how we estimated $T$. This, together with the fact $R\ll S$ for $V$ large enough, gives us
\begin{align*}
\frac{1}{V}\ln\bigg(a_{N}\frac{e^{dV(c/A)^{1/d}}}{e^{dV(c/A(1+\alpha))^{1/d}}}\bigg)<\frac{1}{V}\ln\bigg(\frac{R+S}{T}\bigg)<\frac{1}{V}\ln\bigg(b_N\frac{e^{dV(c/A)^{1/d}}}{e^{dV(c/A(1-\alpha))^{1/d}}}\bigg)
\end{align*}
where $a_N, b_N$ are some constants depending on $N$. Thus
\begin{align*}
\frac{\ln a_N}{V} + d((c/A)^{1/d}-(c/A(1+\alpha))^{1/d})<\frac{1}{V}\ln\bigg(\frac{R+S}{T}\bigg)<\frac{\ln b_N}{V} + d((c/A)^{1/d}-(c/A(1-\alpha))^{1/d}).
\end{align*}
Now for an $\epsilon>0$, pick $\alpha$ such that 
\[
\ds\max\{|d((c/A)^{1/d}-(c/A(1+\alpha))^{1/d})|,|d((c/A)^{1/d}-(c/A(1-\alpha))^{1/d})|\}<\frac{\epsilon}{2}.
\] 
Then we pick $V$ large enough so that 
\[\ds\max\{|a_N/V\big|,|b_N/V|\}<\frac{\epsilon}{2},\]then
\begin{align*}
\bigg|\frac{1}{V}\ln\bigg(\frac{R+S}{T}\bigg)\bigg|<\frac{\epsilon}{2}+\frac{\epsilon}{2}=\epsilon.
\end{align*}
Thus $\ds\frac{1}{V}\ln\bigg(\frac{R+S}{T}\bigg)\to 0$ as $V\to \infty$ and we have finished the case $m=1$.

For $m>1$, we have 
\begin{align*}
\frac{1}{V}\ln\bigg(\sum_{x\in \Z^m}  \frac{(V^dc)^{x}}{\prod_{i=1}^m\theta_i(1)\cdots\theta_i(x_i)}\bigg) &= \frac{1}{V} \ln\bigg( \prod_{i=1}^m\sum_{x_i\in \Z}\frac{(V_i^{d_ic_i})^{x_i}}{\theta_i(1)\cdots\theta_i(x_i)}\bigg)\\
&=\frac{1}{V} \sum_{i=1}^m\ln\bigg( \sum_{x_i\in \Z}\frac{(V_i^{d_ic_i})^{x_i}}{\theta_i(1)\cdots\theta_i(x_i)}\bigg)\\
&\sim \frac{1}{V} \sum_{i=1}^m \ln\bigg(\sum_{x_i\in \Z}  \frac{(V_i^{d_ic_i})^{x}}{A_i^{x_i}(x_i!)^{d_i}}\bigg) \\
&=\frac{1}{V} \ln\bigg(\prod_{i=1}^m \sum_{x_i\in \Z}  \frac{(V_i^{d_ic_i})^{x}}{A_i^{x_i}(x_i!)^{d_i}}\bigg)\\
&=\frac{1}{V} \ln\bigg(\sum_{x\in \Z^m}  \frac{(V^dc)^{x}}{\prod_{i=1}^mA_i^{x_i}(x_i!)^{d_i}}\bigg).
\end{align*}
Note that here the asymptotic analysis still holds after finite addition, since the asymptotic relation we prove for the case $m=1$ is slightly stronger than the usual definition of asymptotic. 
\end{proof}
 \bibliographystyle{plain}
\bibliography{NM}

\end{document}